\documentclass[11pt,reqno,final]{article}
\usepackage{otitis}
\hypersetup{pdfpagemode=UseOutlines,colorlinks=false,
pdfpagelayout=SinglePage,pdfstartview=FitH,bookmarksopen=true}

\begin{document}
%\texttt{\jobname.tex}

\title{On the  $\pd$- and $\barpd$-Operators  \\of a Generalized Complex Structure
}
\author{
 \textsc{Zhuo Chen} \\
{\small Peking University, Department of Mathematics} \\
 {\small   Beijing 100871, China.} \\
{\small
\href{mailto:chenzhuott@gmail.com}{\texttt{chenzhuott@gmail.com}}}}

\date{}
%\date{\texttt{\jobname.tex}}
%\subjclass{}
\maketitle

\begin{abstract}

In this note, we prove that the   $\pd$- and $\barpd$-operators
introduced by Gualtieri for a generalized complex structure coincide
with the  $\bdees$- and $\bdel$-operators introduced by Alekseev-Xu
for Evens-Lu-Weinstein modules  of a Lie bialgebroid.

\end{abstract}

 \tableofcontents

%\section*{Introduction}
%\addcontentsline{toc}{section}{Introduction}

\section{Introduction}
Generalized complex structures
\cites{MR2013140,GualtieriThesis,MR2375780} have been extensively
studied recently due to their close connection with mirror symmetry.
They include both symplectic and complex structures as extreme
cases.  Gualtieri defined the   $\pd$-  and $\barpd$-operators for
any twisted generalized complex structure in the same way the $\pd$-
and  $\barpd$-operators are defined in complex geometry
\cites{GualtieriThesis,math/0703298}.

A Lie bialgebroid, as introduced by Mackenzie \& Xu, is a pair of
Lie algebroids $(A, A^*)$ satisfying some compatibility condition
\cites{MR1262213, MR1362125}. They appear naturally in many places
in Poisson geometry. In \cite{AlekseevXu}, two differential
operators $\bdees$, $\bdel$ were introduced for Evens-Lu-Weinstein
modules of a Lie bialgebroid as follows.

Consider a pair of (real or complex) Lie algebroid structures on a
vector bundle $A$ and its dual $A^*$. Assume that the line bundle
(real or complex) $\module=(\wedge^{top} A^*\otimes\wedge^{top}
T^*M)^{\thalf}$ exists, then $\module$ is a module over $A^*$, as
discovered by Evens, Lu and Weinstein \cite{MR1726784}. The Lie
algebroid structures of $A^*$ and $A$ induce two natural
differential operators $\bdees:\sections{\wedge^k
A\otimes\module}\to
\sections{\wedge^{k+1} A\otimes\module}$ and
$\bdel:\sections{\wedge^k A\otimes\module} \to\sections{\wedge^{k-1}
A\otimes\module}$ (see Equations \eqref{6} to \eqref{11}).

Since a generalized complex structure $\JJ$ induces a (complex) Lie
bialgebroid $(\nolL,\olL)$, where $\nolL$ and $\olL$ are,
respectively, the $(+i)$ and $(-i)$-eigenspaces  of $\JJ$,  it is
tempting to investigate the relations between the operators $\pd$,
$\barpd$, $\bdees$, $\bdel$. In this note, we show that $\pd$ and
$\barpd$ essentially coincide with $\bdees$ and $\bdel$
respectively, under some natural isomorphisms.
\paragraph{Acknowledgments}
 The author  thanks   Camille Laurent-Gengoux, Yi Lin and
 Ping Xu for  useful discussions and comments. This research was
 partially supported by NSFC, grant 10871007.

\section{Courant algebroids and Lie bialgebroids}
In this article, all vector bundles are complex vector bundles.
Likewise,  Lie algebroids are always complex Lie algebroids.

A (complex) Courant algebroid consists of a  vector bundle $\pi:E\to
M$, a nondegenerate pseudo-metric $\ip{\cdot}{\cdot}$ on the fibers
of $\pi$, a bundle map $\rho:~E\to \TCM$, called anchor, and a
$\CC$-bilinear operation $\db{}{}$ on $\sections{E}$ called Dorfman
bracket, which, for all $f\in\cinf{M,\CC}$ and
$z_1,z_2,z_3\in\sections{E}$ satisfy the relations
\begin{align}
& \db{z_1}{(\db{z_2}{z_3})}=\db{(\db{z_1}{z_2})}{z_3}+\db{z_2}{(\db{z_1}{z_3})}; \label{g1} \\
& \rho(\db{z_1}{z_2})=\lb{\rho(z_1)}{\rho(z_2)};  \\
& \db{z_1}{fz_2}=\big(\rho(x)f\big)z_2+f(\db{z_1}{z_2});  \\
& \db{z_1}{z_2}+\db{z_2}{z_1}=2\DD\ip{z_1}{z_2}; \label{g4} \\
& \db{\DD f}{z_1}=0;  \\
&
\rho(z_1)\ip{z_2}{z_3}=\ip{\db{z_1}{z_2}}{z}+\ip{z_2}{\db{z_1}{z_3}}
,
\end{align} where $\DD:\cinf{M,\CC}\to\sections{E}$ is the
$\CC$-linear map defined by $\ip{\DD f}{z_1}=\thalf\rho(z_1)f$.

The symmetric part of the Dorfman bracket is given by \eqref{g4}.
The Courant bracket is defined as the skew-symmetric part
$\cb{z_1}{z_2}=\thalf(\db{z_1}{z_2}-\db{z_2}{z_1})$ of the Dorfman
bracket. Thus we have the relation
$\db{z_1}{z_2}=\cb{z_1}{z_2}+\DD\ip{z_1}{z_2}$.

The definition of a Courant algebroid can be rephrased
using the Courant bracket instead of the Dorfman bracket
 \cite{math/9910078}.

A Dirac structure is a smooth subbundle $A\to M$ of
the Courant algebroid $E$, which is maximal isotropic
 with respect to the pseudo-metric and whose
 space of sections is closed under (necessarily both) brackets.
Thus a Dirac structure inherits a canonical   Lie algebroid
structure \cite{MR1472888}.

Let $A\to M$ be a vector bundle. Assume that $A$ and its dual $A^*$
both carry a Lie algebroid structure with anchor maps $\anchor:A\to
{\TCM}$ and $\anchors:A^*\to {\TCM}$, brackets on sections
$\sections{A}\otimes_{\CC}\sections{A}\to\sections{A}: X\otimes
Y\mapsto\ba{X}{Y}$ and
$\sections{A^*}\otimes_{\CC}\sections{A^*}\to\sections{A^*}:
\theta\otimes \xi\mapsto\bas{\theta}{\xi}$, and differentials
$\dee:\sections{\wedge^{\bullet}A^*}\to\sections{\wedge^{\bullet+1}A^*}$
and
$\dees:\sections{\wedge^{\bullet}A}\to\sections{\wedge^{\bullet+1}A}$.

This pair of Lie algebroids $(A,A^*)$ is a Lie bialgebroid (or Manin
triple) \cites{MR1362125,MR1746902,MR1262213} if $\dees$ is a
derivation of the Gerstenhaber algebra $(\sections{\wedge\graded
A},\wedge,\ba{\cdot}{\cdot})$ or, equivalently, if $\dee$ is a
derivation of the Gerstenhaber algebra $(\sections{\wedge\graded
A^*},\wedge,\bas{\cdot}{\cdot})$. The link between Courant and Lie
bialgebroids is given by the following

\begin{thm}[\cite{MR1472888}]
\label{Thm:double} There is a 1-1 correspondence between Lie
bialgebroids and pairs of transversal Dirac structures in a Courant
algebroid.
\end{thm}

More precisely, if the pair $(A,A^*)$ is a Lie bialgebroid, then the
vector bundle $A\oplus A^*\to M$ together with the pseudo-metric
\begin{equation}
\label{3}  \ip{X_1+\xi_1}{X_2+\xi_2}=
\thalf\big(\xi_1(X_2)+\xi_2(X_1)\big) ,
\end{equation}
the anchor map $\rho=\anchor+\anchors$ (whose dual is given by $\DD
f=\dee f+\dees f$ for $f\in\cinf{M,\CC}$) and the Dorfman bracket
\begin{equation}\label{4}
\db{(X_1+\xi_1)}{(X_2+\xi_2)}=
\big(\ba{X_1}{X_2}+\ld{\xi_1}X_2-\ii{\xi_2}(\dees X_1)\big)+
\big(\bas{\xi_1}{\xi_2}+\ld{X_1}\xi_2-\ii{X_2}(\dee\xi_1)\big)
\end{equation}
is a Courant algebroid of which $A$ and $A^*$
are transverse Dirac structures.
It is called the double of the Lie bialgebroid $(A,A^*)$.
Here $X_1,X_2$ denote arbitrary sections of $A$
and $\xi_1,\xi_2$ arbitrary sections of $A^*$.

An important example is that when $A={\TCM}$ is the tangent bundle
of a manifold $M$ and $A^*=\TsCM$ takes the trivial Lie algebroid
structure. Then ${\TCM}\oplus \TsCM$ has the standard Courant
algebroid structure. As observed by Severa  and Weinstein in
\cite{MR2023853}, the Dorfman bracket on ${\TCM}\oplus \TsCM$ can be
twisted by a closed $3$-form $H\in Z^3(M)$:
\begin{eqnarray*}
 &&\dbH{ (x_1+\eta_1) }{ (x_2+\eta_2) }  \\\nonumber
 &=&
\db{ (x_1+\eta_1) }{ (x_2+\eta_2) }+\ii{x_2}\ii{x_1}H\\
&=&[x_1,x_2]+\ld{x_1}\eta_2-\ld{x_2}\eta_1+\thalf d
\duality{\eta_1}{x_2}+\ii{x_2}\ii{x_1}H.
\end{eqnarray*}
And $\dbH{}{}$ defines a Courant algebroid structure on
${\TCM}\oplus \TsCM$, using the same inner product and anchor. The
corresponding Courant bracket is also twisted:
\begin{eqnarray}\nonumber
 &&\cb{ x_1+\eta_1 }{ x_2+\eta_2 }_H \\\nonumber
 &=&
\cb{ x_1+\eta_1 }{ x_2+\eta_2
}+\ii{x_2}\ii{x_1}H\\\label{twistedbracketH}
&=&[x_1,x_2]+\ld{x_1}\eta_2-\ld{x_2}\eta_1+\thalf d
(\duality{\eta_1}{x_2}-\duality{\eta_2}{x_1})+\ii{x_2}\ii{x_1}H.
\end{eqnarray}

\section{Clifford modules and Dirac generating operators}
\label{subsec:CliffordandDiracAAs}

 Let $V$ be a vector space of
dimension $r$ endowed with a nondegenerate symmetric bilinear form
$\ip{\cdot}{\cdot}$. Its Clifford algebra $\clifford{V}$ is defined
as the quotient of the tensor algebra $\oplus_{k=0}^r V^{\otimes r}$
by the relations $x\otimes y+y\otimes x=2\ip{x}{y}$ ($x,y\in V$). It
is naturally an associative $\ZZ_2$-graded algebra. Up to
isomorphisms, there exists a unique irreducible module $S$ of
$\clifford{V}$, called spin representation \cite{MR1636473}. The
vectors of $S$ are called spinors.

\begin{ex}\label{translagr}
Let $W$ be a vector space of dimension $r$. We can endow $V=W\oplus
W^*$ with the nondegenerate pairing defined in the same fashion as
in Eq. (\ref{3}). The representation of $\clifford{V}$ on
$S=\oplus_{k=0}^r \wedge^k W$ defined by $u\cdot w=u\wedge w$ and
$\xi\cdot w=\ii{\xi}w$, where $u\in W$, $\xi\in W^*$ and $w\in S$,
is the spin representation. Note that $S$ is $\ZZ$- and thus also
$\ZZ_2$-graded.
\end{ex}

Now let $\pi:E\to M$ be a vector bundle endowed with a nondegenerate
pseudo-metric  $\ip{\cdot}{\cdot}$ on its fibers and let
$\clifford{E}\to M$ be the associated
 bundle of Clifford algebras.
Assume there exists a smooth vector bundle $S\to M$ whose fiber
$S_m$ over a point $m\in M$ is the spin module of the Clifford
algebra $\clifford{E}_m$. Assume furthermore that $S$ is
$\ZZ_2$-graded: $S=S^0\oplus S^1$.

An operator $O$ on $\sections{S}$ is called even (or of degree~0) if
$O(S^i)\subset S^{i}$ and odd (or of degree~1) if $O(S^i)\subset
S^{i+1}$. Here $i\in\ZZ_2$.

\begin{ex} If the vector bundle $E$ decomposes as
the direct sum $A\oplus A^*$ of two transverse Lagrangian subbundles
as in Example~\ref{translagr}, then $S=\wedge A$. The multiplication
by a function $f\in\cinf{M,\CC}$ is an even operator on
$\sections{S}$ while the Clifford action of a section
$e\in\sections{E}$ is an odd operator on $\sections{S}$.
\end{ex}

If $O_1$ and $O_2$ are operators of degree~$d_1$ and $d_2$
respectively, then their commutator is the operator
$\lb{O_1}{O_2}=O_1\rond O_2-(-1)^{d_1 d_2}O_2\rond O_1$.

\begin{defn}[\cite{AlekseevXu}]
\label{Def:DiracGTR} A Dirac generating operator for $(E,\ip{~}{~})$
is an odd operator $\dirac$ on $\sections{S}$ satisfying the
following properties:
\begin{enumerate}
\item \label{conda} For all $f\in\cinf{M,\CC}$,
$\lb{\dirac}{f}\in\sections{E}$. This means that the operator
$\lb{\dirac}{f}$ is the Clifford action of some section of $E$.
\item \label{condb} For all $z_1,z_2\in\sections{E}$,
$\lb{\lb{\dirac}{z_1}}{z_2}\in\sections{E}$.
\item \label{condc} The square of $\dirac$ is the
multiplication by some function on $M$: that is
$\dirac^2\in\cinf{M,\CC}$.
\end{enumerate}
\end{defn}

Note that ``deriving operators'' of \cite{MR2103012} do not require
the assumption \ref{condc}.

\begin{thm}[\cite{AlekseevXu}]
%Let $\pi:E\to M$ be a vector bundle endowed with a nondegenerate pseudo-metric $\ip{\cdot}{\cdot}$ on its fibers and let $\clifford{E}\to M$ be the associated bundle of Clifford algebras.
%Assume there exists a smooth $\ZZ_2$-graded vector bundle $S\to M$ whose fiber $S_m$ over a point $m\in M$ is the spin module of the Clifford algebra $\clifford{E}_m$.
Let $\dirac$ be a Dirac generating operator for a vector bundle
$\pi:E\to M$. Then there is a canonical Courant algebroid structure
on $E$. The anchor $\rho:E\to {\TCM}$ is defined by
$\rho(z)f=2\ip{\lb{\dirac}{f}}{z}=\lb{\lb{\dirac}{f}}{z}$, while the
Dorfman bracket reads $\db{z_1}{z_2}=\lb{\lb{\dirac}{z_1}}{z_2}$.
\end{thm}

We follow the same setup as in \cite{AlekseevXu} and
\cite{ChenStienon}.

Let $(A,\ba{\cdot}{\cdot},\anchor)$ and
$(A^*,\bas{\cdot}{\cdot},\anchors)$ be a pair of Lie algebroids,
where $A$ is of rank-$r$ and the base manifold $M$ is of
dimension-$m$. The line bundle $\wedge^r A^*\otimes\wedge^m {\TsCM}$
is a module over the Lie algebroid $A^*$ \cite{MR1726784}: a section
$\alpha\in\sections{A^*}$ acts on $\sections{\wedge^r
A^*\otimes\wedge^m {\TsCM}}$ by
\begin{eqnarray}\nonumber
&&\alpha\cdot(\alpha_1\wedge\cdots\wedge\alpha_r\otimes\mu)
\\\label{6}
&=& \sum_{i=1}^r
\big(\alpha_1\wedge\cdots\wedge\bas{\alpha}{\alpha_i}
\wedge\cdots\wedge\alpha_r\otimes\mu\big)
 +\alpha_1\wedge\cdots\wedge\alpha_r\otimes\ld{\anchors(\alpha)}\mu .
 \end{eqnarray}
If it exists, the square root $\module=(\wedge^r A^*\otimes\wedge^m
{\TsCM})^{\thalf}$ of this line bundle is also a module over $A^*$.
One can thus define a differential operator
\begin{equation}\label{7}
\bdees:\sections{\wedge^k A\otimes\module} \to\sections{\wedge^{k+1}
A\otimes\module} .\end{equation} Similarly, $\explicit$ is ---
provided it exists --- a module over $A$. Hence we obtain a
differential operator
\begin{equation}\label{8}
\sections{\wedge^k A^*\otimes\explicit}\to
\sections{\wedge^{k+1} A^*\otimes\explicit} .\end{equation}
But the isomorphisms of vector bundles
\begin{equation}\label{9}
\wedge^k A^*\isomorphism \wedge^k A^* \otimes
\wedge^{r-k}A^*\otimes\wedge^{r-k}A\isomorphism\wedge^{r-k}
A\otimes\wedge^r A^* \end{equation} and
\begin{equation}\label{10}
\wedge^r A^*\otimes\explicit\isomorphism (\wedge^r
A^*\otimes\wedge^m {\TsCM})^{\thalf} \end{equation} imply that
\begin{multline} \label{m1}
\wedge^k A^*\otimes\explicit \isomorphism
\wedge^{r-k}A\otimes\wedge^r A^*\otimes \explicit \\ \isomorphism
\wedge^{r-k}A\otimes (\wedge^r A^*\otimes\wedge^m {\TsCM})^{\thalf}
.\end{multline} Therefore, one ends up with a differential operator
\begin{equation}\label{11} \bdel:\sections{\wedge^k A\otimes\module}
\to\sections{\wedge^{k-1} A\otimes\module} .\end{equation}

%Let $M$ be a smooth manifold of dimension $m$ and let $A\to M$ be a smooth vector bundle of rank $n$. Assume that $A$ and its dual $A^*$ both carry a Lie algebroid structure with anchor maps $\anchor:A\to {\TCM}$ and $\anchors:A^*\to {\TCM}$, and bracket on sections $\sections{A}\otimes_{\CC}\sections{A}\to\sections{A}:u\otimes v\mapsto\ba{u}{v}$ and $\sections{A^*}\otimes_{\CC}\sections{A^*}\to\sections{A^*}:\theta\otimes \eta\mapsto\bas{\theta}{\eta}$.

The following theorem is  proved in \cite{ChenStienon}.

\begin{thm}\label{Thm:A}
The pair of Lie algebroids $(A,A^*)$ is a Lie bialgebroid if, and
only if, $\bdirac^2\in\cinf{M,\CC}$, i.e. the square of the operator
$\bdirac=\bdees+\bdel$: $\sections{\wedge A\otimes\module}\to
\sections{\wedge A\otimes\module}$ is the multiplication by some function
$\fsmile\in\cinf{M,\CC}$. Moreover $\bdirac^2_*=\fsmile$, where
$\bdiracs=\bdee+\bdels$ is defined analogously to $\bdirac$ by
exchanging the roles of $A$ and $\As$.
\end{thm}

\section{Generalized complex geometry}\label{subsec:GCM} In this
section, we fix a real $2n$-dimensional manifold $M$ and denote the
tangent and cotangent bundle of $M$ by $T$ and $\Ts$, respectively.
Let $\TC$ and $\TsC$ be, respectively, the complexification of $T$
and $\Ts$. The first vital ingredient in $\TC\oplus \TsC$ is the
natural pairing:
\begin{equation}\label{naturalpairing}
\ip{x_1+\eta_1}{x_2+\eta_2}=\half(\duality{x_1}{\eta_2}
+\duality{x_2}{\eta_1}),\quad\quad \forall x_i\in \TC, \eta_i\in
\TsC.
\end{equation}
Thus we have the Clifford algebra $\clifford{\TC\oplus \TsC}$, which
acts on the spinor bundle $\GCCM\defbe \oplus_{i=0}^{2n} \wedge^i
\TsC$ via
$$
(x+\eta)\cdot \rho= \ii{x}\rho+\eta\wedge\rho,\quad\forall
\rho\in\GCCM.
$$

Introduce a $\CC$-linear map $\cdot^T:~\GCCM\lon \GCCM$   by
$$(\eta_1\wedge\cdots\wedge \eta_j)^T
=\eta_j\wedge\cdots\wedge \eta_1.$$
  The Mukai
pairing $\Mukai{~}{~}$: $\GCCM\times \GCCM\lon \wedge^{2n}(\TsC)$ is
defined by
$$
\Mukai{\chi}{\omega}=[\chi^T\wedge \omega]^{2n},
$$ where $[~]^{2n}$ indicates the top degree  component of the
product.
  Explicitly, if   $\chi=\sum_{i=0}^{2n} \chi_i$,
$\omega=\sum_{i=0}^{2n} \omega_i$, where $\chi_i,\omega_i\in
\wedge^i \TsC$, then
$$
\Mukai{\chi}{\omega}=\sum_{i=0}^{2n}(-1)^{\frac{i(i-1)}{2}}\chi_i\wedge
\omega_{2n-i}.
$$

The following properties are standard \cite{GualtieriThesis}:
\begin{equation}\label{Mukasymmetry}
\Mukai{\chi}{\omega}=(-1)^n \Mukai{\omega}{\chi},
\end{equation}
\begin{equation}\label{Mukaequivariant}
\Mukai{\phi\wedge\chi}{\omega}+\Mukai{\chi}{\phi\wedge\omega}=0,
\end{equation}
for all $\chi,\omega\in \GCCM$, $\phi\in \wedge^2 \TsC $.

Consider a real, closed $3$-form $H\in Z^3(M)$ which induces a
twisted differential operator
$$
\dH=d+  H\wedge( \cdot).
$$

\begin{defn}\cites{GualtieriThesis,MR2265463} A twisted generalized complex structure with respect to
$H$ is determined by any of the following three equivalent objects:
\begin{itemize}
\item[(i)] A real automorphism $\JJ$ of $T\oplus \Ts$, which squares
to $-1$,   is orthogonal with respect to the natural pairing
(\ref{naturalpairing}), and has vanishing Nijenhuis tensor, i.e.,
for all $z_1,z_2\in \sections{T\oplus \Ts}$,
$$
N(z_1,z_2)\defbe -\cb{\JJ z_1}{\JJ z_2}_H+\JJ\cb{\JJ z_1}{z_2}_H
+\JJ\cb{ z_1}{\JJ z_2}_H+\cb{ z_1}{z_2}_H=0.
$$
Here $\cb{}{}_H$ is the twisted Courant bracket defined in Eq.
(\ref{twistedbracketH}).
\item[(ii)] A twisted Dirac structure $L\subset \TC\oplus \TsC$ with
respect to $H$, which satisfies $L\cap \overline{L}=\set{0}$.
\item[(iii)]A line subbundle $N$ of $\GCCM=\wedge^{\bullet}(\TsC)$ generated at each point by a
form   $u$, such that
$$\nolL=\set{X\in \TC\oplus \TsC~|~
X\cdot u=0 }
$$is maximally isotropic,
  $\Mukai{u}{\bar{u}}  \neq 0 $, and
$$\dH u=e\cdot u,$$
for some $e\in\sections{T \oplus \Ts }$.
\end{itemize}
\end{defn}
The line bundle $N $ in (iii) is called the pure spinor line bundle
corresponding to $\nolL$.
\begin{rmk}\rm
 Also see \cite{AlekseevXu} for the relation between Dirac
 structures and Dirac generating operators.
\end{rmk}
As a generalization of the usual    $\pd$- and $\barpd$-operators in
complex geometry, Gualtieri introduced the  $\pd$- and
$\barpd$-operators for any twisted generalized complex structure
\cite{GualtieriThesis}. We recall its construction briefly below.

Let $\JJ$ be   a twisted generalized complex structure and
$\nolL\subset \TC\oplus \TsC$ be the $+i$-eigenspace  of $\JJ$.
$\nolL$ is a twisted Dirac structure and satisfies $\nolL\cap \olL
=\set{0}$. We will regard $\olL=\nolL^*$ by defining the canonical
pairing between $\nolL$ and $\olL $ by
\begin{equation}\label{Eqt:dualityLolL}
\duality{X}{\theta}=2\ip{X}{\theta},\quad \forall ~ X\in \nolL, e\in
\olL .
\end{equation}

Set $N_0=N$, $N_k=\wedge^k\olL \cdot N$ ($k=0,\cdots,2n$). Then
$\olN_k=N_{2n-k}$ and especially, $N_{2n}=\olN$ is the pure spinor
of $\olL $, and we have a decomposition $$\GCCM=N_0\oplus N_1\oplus
\cdots \oplus N_{2n}.$$

It is proved that one can decompose \cite{GualtieriThesis}
$$
\dH=\pd+\barpd.
$$
Here $\pd:~\sections{N_{\bullet}}\lon
\sections{N_{\bullet-1}}$ and $\barpd:~\sections{N_{\bullet}}\lon \sections{N_{\bullet+1}}$
(or $~\sections{\olN_{\bullet}}\lon \sections{\olN_{\bullet-1}}$ )
 are defined by:
$$
\pd (n_k) \defbe pr_{N_{k-1}}(\dH n_k),\quad\barpd (n_k) \defbe
pr_{N_{k+1}}(\dH n_k),\quad\forall~ n_k\in
\sections{N_k}.
$$

\section{The main theorem}\label{Sec:mainthm}
Given a generalized complex structure $\JJ$ as above, it is clear
that $(\nolL,\olL)$ is a Lie bialgebroid (regarding $\nolL^*=\olL$
by Eq. (\ref{Eqt:dualityLolL})). We   prove that the operators
$\bdees$ and $\bdel$ for this particular situation are essentially
$\pd $ and $\barpd$ (Theorem \ref{Thm:Main}).

We continue the notations in Section \ref{subsec:GCM}  and let   $u$
be a nowhere vanishing local section of $N$.
%In \cite{MR2013140},
%Hitchin calls this kind of generalized complex structure the
%generalized Calabi-Yau structures.
Assume that $\V\in
\sections{\wedge^{2n}\nolL}$ satisfies $ \V\cdot \olu=u$. Hence $
\olV\cdot u=\olu$ and
$$
\duality{\V}{\olV}u=(-1)^n\V\cdot \olV\cdot u=(-1)^n u,\quad
\duality{\V}{\olV}=(-1)^n.
$$
So   the dual section of $\V$ is given by $\Omega=(-1)^n\olV\in
\sections{\wedge^{2n}\olL}$.

\begin{prop}(\cite[Prop. 2.22]{GualtieriThesis}, see also
 \cite[III.3.2]{MR1636473})
The line bundle $\module=(\wedge^{2n} \olL\otimes\wedge^{2n}
\TsC)^{\thalf}$ and  $N_{2n}=\olN$ are canonically isomorphic. The
isomorphism can be explicitly described by the following:
\begin{eqnarray}\label{canonicaliso}
 \olN\otimes\olN&\lon& \module\\\nonumber
  \omega_1\otimes \omega_2&\mapsto&\Omega\otimes \Mukai{V\cdot
\omega_1}{\omega_2},\quad~\forall~\omega_1,\omega_2\in \olN.
\end{eqnarray}
The isomorphism (\ref{canonicaliso}) does not depend on the choice
of $u$ and $V$.
\end{prop}

From now on we will identify $\olN$ with $(\wedge^{2n}
\olL\otimes\wedge^{2n} \TsC)^{\thalf}$.
  As a consequence of
the $\olL$-module structure on $(\wedge^{2n} \olL\otimes\wedge^{2n}
\TsC)^{\thalf}$, we have two differential operators (see Eq.
(\ref{7}), (\ref{11})):
\begin{align*}
&\bdees:~(\wedge^\bullet \nolL)\otimes\olN\lon
(\wedge^{\bullet+1} \nolL)\otimes\olN ,\\
&\bdel:~(\wedge^\bullet \nolL)\otimes\olN\lon (\wedge^{\bullet-1}
\nolL)\otimes\olN.
\end{align*}

It is also shown in \cite{GualtieriThesis} that
 $ (\wedge^k\olL
)\otimes N \cong N_k $ and $ (\wedge^k\nolL )\otimes \olN \cong
\olN_k$ respectively by the following two isomorphisms:
$$
I:~ (\wedge^k\olL )\otimes N \lon N_k,\quad W\otimes \pu\lon W\cdot
\pu,\quad\forall W\in \wedge^k\olL, \pu\in N,
$$
$$\olI:~(\wedge^k\nolL )\otimes \olN \lon
\olN_k,\quad X\otimes \olpu\lon X\cdot \olpu,\quad\forall X\in
\wedge^k\nolL, \pu\in N.
$$

Our main theorem is
\begin{thm}\label{Thm:Main}The following two diagrams are commutative.
\begin{equation}\label{Digram1}
\xymatrix{ {(\wedge^k\nolL )\otimes \olN} \ar[r]^{ \bdees } \ar[d]_{I} & {(\wedge^{k+1}\nolL )\otimes \olN} \ar[d]^{I} \\
\olN_k \ar[r]_{\pd} & \olN_{k+1}, }
\end{equation}
\begin{equation}\label{Digram2}
\xymatrix{ {(\wedge^k\nolL )\otimes \olN} \ar[r]^{ \bdel } \ar[d]_{\olI} & {(\wedge^{k-1}\nolL )\otimes \olN} \ar[d]^{
\olI} \\
\olN_k \ar[r]_{\barpd} & \olN_{k-1}, }
\end{equation}
\end{thm}
The proof will be deferred to Section \ref{Sec:ProofofMain}.

In \cite{math/0703298}, Gualtieri constructed an
   $\nolL$-module structure on $N$ and an $\olL$-module structure on
   $\olN$, respectively by
\begin{eqnarray}\label{Eqt:LieDerXu}
\Rep_{X}\pu &\defbe& X\cdot \dH \pu=X\cdot\barpd \pu,
\\\label{Eqt:LieDerWolu} \Rep_{W}\olpu &\defbe& W\cdot \dH
\olpu=W\cdot\pd \olpu,\quad
\end{eqnarray}
$\forall \pu\in
\sections{N},\  X\in
\sections{\nolL},\  W\in
\sections{\olL}$.

As a special situation of $k=0$ in Diagram (\ref{Digram1}), we have
\begin{prop}\label{Prop:TwoMudlestrEquivalent}The above $\olL$-module structure
defined by Eq. (\ref{Eqt:LieDerWolu}) coincides with the
$\olL$-module structure defined by Eq. (\ref{6}), under the
isomorphism (\ref{canonicaliso}).
\end{prop}

\section{Modular cocycles of Lie algebroids
}\label{subsec:ModularLiealgebroids}

 In this section we
establish a list of important identities valid in any   Lie
bialgebroid  $(A,A^*)$ and generalized complex structure, which are
subsequently used in Section \ref{Sec:ProofofMain} to prove the
statements of Section \ref{Sec:mainthm}.

 We continue the
assumptions in Section \ref{subsec:CliffordandDiracAAs}: let
$(A,\ba{\cdot}{\cdot},\anchor)$ and
$(A^*,\bas{\cdot}{\cdot},\anchors)$ be a pair of rank-$r$ Lie
algebroids over dimension-$m$ base manifold $M$.

Assume there exists a volume form $s\in\sections{\wedge^m {\TsCM}}$
and a nowhere vanishing section $\Omega\in\sections{\wedge^r A^*}$
so that $\module$ is the trivial line bundle over $M$. And let
$V\in\sections{\wedge^r A}$ be the section dual to $\Omega$:
$\duality{\Omega}{V}=1$. These induce two bundle isomorphisms:
\begin{gather}
\Omega\diese:\wedge^k A\to\wedge^{r-k}A^*:X\mapsto \ii{X}\Omega ,
\label{g10} \\
V\diese:\wedge^k A^*\to\wedge^{r-k}A:\xi\mapsto \ii{\xi}V
,\label{g11}
\end{gather}
which are essentially inverse to each other:

\begin{gather}\label{diverse}
(V^\sharp\circ\Omega^\sharp)(X)=(-1)^{k(r-1)}X,\qquad\forall X\in
\wedge^k A; \\
\label{residents} (\Omega^\sharp\circ V^\sharp)(\varphi)
=(-1)^{k(r-1)}\varphi,\qquad\forall\varphi\in \wedge^k\As.
\end{gather}

Consider the operator $\del$ dual to $\dee$ with respect to
$V\diese$:
\begin{equation}\label{12} \xymatrix{ \sections{\wedge^k A^*}
\ar[r]^{V\diese} \ar[d]_{-(-1)^k\dee} &
\sections{\wedge^{r-k} A} \ar[d]^{\del} \\
\sections{\wedge^{k+1} A^*} \ar[r]_{V\diese} &
\sections{\wedge^{r-k-1} A}, } \end{equation}
or \begin{equation}\label{13} -V\diese\dee\alpha=(-1)^k\del
V\diese\alpha, \quad \forall\alpha\in\sections{\wedge^k A^*},
\end{equation}
which, due to \eqref{diverse} and \eqref{residents}, can be
rewritten as
\begin{equation}\label{13bis}
\Omega\diese \del \beta=(-1)^l \dA \Omega\diese \beta,\quad\forall
\beta\in \sections{\wedge^l A}.
\end{equation}

The operator $\del$ is a Batalin-Vilkovisky operator for the Lie
algebroid $A$ \cites{MR1764439,MR837203,MR1675117,MR2182214}.
Similarly, we  have the operator $\dels$ dual to $\dees$:
\begin{equation*}\label{14} \xymatrix{ \sections{\wedge^{r-k} A}
\ar[d]_{ (-1)^k\dees} &
\sections{\wedge^{ k} A^*} \ar[l]_{V\diese} \ar[d]^{\dels} \\
\sections{\wedge^{r-k+1} A} &
\sections{\wedge^{ k-1} A^*} \ar[l]^{V\diese}, } \end{equation*}
or \begin{equation*}\label{15}  \dees V\diese\alpha=(-1)^k
V\diese\dels \alpha, \quad \forall\alpha\in\sections{\wedge^{ k}
A^*}.
\end{equation*}

According to \cite{MR1726784}, there exists a unique
$X_0\in\sections{A}$ such that
\begin{equation}\label{17} \ld{\theta}(\Omega\otimes s)
=(\ld{\theta}\Omega)\otimes s+\Omega\otimes(\ld{\anchors(\theta)} s)
=\duality{X_0}{\theta}\Omega\otimes s, \quad\forall
\theta\in\sections{A^*} .
\end{equation}
Similarly, there exists a unique $\xi_0\in\sections{A^*}$ such that
\begin{equation}\label{18} \ld{X}(s\otimes V)=(\ld{\anchor(X)}s)\otimes V
+s \otimes(\ld{X}V)=\duality{\xi_0}{X}s\otimes V, \quad\forall
X\in\sections{A} .\end{equation} These sections $X_0$ and $\xi_0$
are called \emph{modular cocycles} and their cohomology classes are
called modular classes \cite{MR1726784}.

A simple computation yields the following formulas, which are given
in \cite{AlekseevXu}.
\begin{prop}With the above notations, the differential operators defined by Eq. (\ref{7}) and
(\ref{11}) are given respectively:
\begin{equation}\label{19}
\bdees(X\otimes l)=(\dees X+\thalf X_0\wedge X)\otimes l
\end{equation} and
\begin{equation}\label{20} \bdel(X\otimes
l)=(-\del X+\thalf \ii{\xi_0}X)\otimes l ,
\end{equation}
for all $X\in\sections{\wedge A}$ and $l\in\sections{\module}$.
\end{prop}

Hence the operator $\bdirac$ in Theorem \ref{Thm:A} reads
\[ \bdirac=\bdees+\bdel=\dees-\del+\thalf(X_0\wedge\cdot~+~\ii{\xi_0})
.\] This construction of Dirac generating operators using modular
cocylces appeared in \cites{AlekseevXu,ChenStienon}.

%\subsection{The modular cocycles of $\nolL$ and $\olL$}
Now we consider a twisted generalized complex structure $\JJ$ on a
$2n$-dimensional manifold $M$ and let $\nolL$ and $\olL$ be
respectively the $+i$ and $-i$-eigenspace of $\JJ$.   Again we
assume that $u$ is a nowhere vanishing local section of $N$, the
pure spinor bundle of $\nolL$.

We need the following basic fact:
\begin{lem}\label{Lem:Existenceofe}\cite{math/0703298}
There exists some $e=x+\eta\in
\sections{\olL }$ such that
\begin{equation}\label{dHueu}
\dH u~=~\barpd u~=~du+H\wedge u~=~e\cdot
u~=~\inserts_{x}u+\eta\wedge u,
\end{equation}
\begin{equation}\label{dHolu}
\dH \olu~=~\pd \olu~=~d\olu+H\wedge \olu~=~\ole\cdot
\olu~=~\inserts_{\overline{x}}\olu+\overline{\eta}\wedge \olu.
\end{equation}

\end{lem}
The main result in this section is the following.
\begin{prop}\label{Thm:modularGC}Let $\V\in
\sections{\wedge^{2n}\nolL}$ such that  $ \V\cdot \olu=u$.
Then the modular cocycle of $\nolL$ with respect to the top form
$\V$ and the volume form $s=\Mukai{u}{\olu}$ is $2e$, where $e\in
\sections{\olL }$ is given by Lemma \ref{Lem:Existenceofe}.

Similarly, the modular cocycle of $\olL$ with respect to
$\Omega=(-1)^n\olV\in
\sections{\wedge^{2n}\olL}$ and $s$ is $2\ole$.
\end{prop}

Before the proof, we need a couple of lemmas. The first one can be
easily verified.
\begin{lem} \label{Lem:XcdotWcdot} For all $W\in
\wedge^j\olL$, $X\in \wedge^i \nolL$,  and  $ i\leq j\leq 2n$, one
has
\begin{equation}\label{Eqt:XWu}
X\cdot W\cdot u=(-1)^{\frac{i(i-1)}{2}}(\ii{X}W)\cdot u\,.
\end{equation}
\end{lem}

Since $\nolL$ is a Lie algebroid and $ \nolL^*=\olL$, we have the
differential
$$d_\nolL:~
\sections{\wedge^{\bullet}\olL }\lon
\sections{\wedge^{\bullet+1}\olL }.$$
Moreover,  we have the following equality:
\begin{eqnarray}\nonumber
\barpd(W\cdot u)&=&(d_\nolL W)\cdot u+(-1)^k W\cdot \barpd
u\\\nonumber &=&(d_\nolL W)\cdot u+(-1)^k (W\wedge e)\cdot
u\\\label{Eqt:pdbarwu} &=&(d_\nolL W+e\wedge W)\cdot u ,\qquad
\forall W\in
\sections{\wedge^k\olL },
\end{eqnarray}
which encodes  the $\nolL$-module structure on $N$ defined by Eq.
(\ref{Eqt:LieDerXu}).

Similarly, one has
\begin{eqnarray}\nonumber
\pd(X\cdot \olu)&=&(d_{\olL } X)\cdot \olu+(-1)^i X\cdot  \pd
\olu\\\label{Eqt:pdXbaru} &=&(d_{\olL } X + \ole\wedge X )\cdot
\olu,\qquad \forall X\in
\sections{\wedge^i {\nolL}}.
\end{eqnarray}
%Here $d_{\olL} X\in \sections{\wedge^{i+1}\nolL }$ is the coboundary
%of $X$ with respect to the Lie algebroid structure of $\olL$.

\begin{lem}For any $X=a+\zeta\in \sections{\nolL}$, we have
\begin{equation}\label{Eqt:LieDeroverlineV}
(\LieDer_{X}\olV)\cdot u=\barpd{(X\cdot \olu)}-\duality{e}{X}\olu,
\end{equation}
\begin{equation}\label{Eqt:LieDerau}
 \LieDer_{a}u=\duality{e}{X} u -(d\zeta+\inserts_a H)\wedge u,
\end{equation}
\begin{equation}\label{Eqt:LieDeraoverlineu}
 \LieDer_{a}\olu=\barpd{(X\cdot \olu)}+(d_{\olL}X )\cdot \olu
-(d\zeta+\inserts_a H)\wedge\olu.
\end{equation}
\end{lem}
\begin{proof}A basic fact is that
\begin{equation}\label{lu0}
0=X\cdot u=\inserts_a u+ \zeta\wedge u,
\end{equation}
for any $X=a+\zeta\in \sections{\nolL}$. Hence
\begin{eqnarray*}
\barpd{(X\cdot \olu)}&=&\barpd(X\cdot  \olV\cdot
u)=\barpd((\inserts_X \olV)\cdot u)
\\&=&
(d_\nolL \inserts_X \olV)\cdot u-(\inserts_X \olV\wedge e)\cdot u
\qquad\text{(by \eqref{Eqt:pdbarwu})}
\\
&=& (d_\nolL \inserts_X \olV+\inserts_X d_\nolL \olV)\cdot
u+(\duality{e}{X}\olV)\cdot u
\\
&=&(\LieDer_{X}\olV)\cdot u+\duality{e}{X}\olu.
\end{eqnarray*}
This proves Eq. (\ref{Eqt:LieDeroverlineV}). For Eq.
(\ref{Eqt:LieDerau}), we have
\begin{eqnarray*}
 \LieDer_{a}u &=& \inserts_a du + d   \inserts_a u\\
 &=&\inserts_a(\inserts_{x}u+\eta\wedge
u-H\wedge u)-d(\zeta\wedge u) \qquad\text{(by \eqref{dHueu} and
\eqref{lu0})}
\\
&=&-\inserts_x\inserts_a u+\duality{a}{\eta}u -\eta\wedge \inserts_a
u-\inserts_a H\wedge u + H\wedge \inserts_a u -d\zeta
\wedge u +\zeta\wedge du\\
&=&\inserts_x(\zeta\wedge u)+\duality{a}{\eta}u +(\eta-H)\wedge
(\zeta\wedge u)\\
&&\quad-\inserts_a H\wedge u   -d\zeta \wedge u +\zeta\wedge
(\inserts_{x}u+\eta\wedge u-H\wedge u) \qquad\text{(by \eqref{dHueu}
and \eqref{Eqt:pdbarwu})}
\\
&=&(\duality{x}{\zeta}+\duality{a}{\eta})u-\inserts_a H\wedge
u-d\zeta \wedge u.
\end{eqnarray*}
To prove Eq. (\ref{Eqt:LieDeraoverlineu}), we observe that, on one
hand
\begin{eqnarray*}
\dH{(X\cdot \olu)}&=&\barpd{(X\cdot \olu)}+\pd{(X\cdot \olu)}\\
&=&\barpd{(X\cdot \olu)}+(d_{\olL}X )\cdot \olu- (X\wedge \ole
)\cdot \olu \qquad\text{(by \eqref{Eqt:pdXbaru})}.
\end{eqnarray*}
On the other hand, we have
\begin{eqnarray*}
\dH{(X\cdot \olu)}&=&d(\inserts_a \olu+ \zeta\wedge \olu)+H\wedge {(X\cdot \olu)}\\
&=&d\inserts_a \olu+ d\zeta\wedge\olu-\zeta\wedge d\olu+H\wedge
{(X\cdot \olu)}
\\
&=&(d\inserts_a\olu+\inserts_a d\olu) +
d\zeta\wedge\olu-(\inserts_a+\zeta\wedge) d\olu+H\wedge {(X\cdot
\olu)}
\\
&=&\LieDer_a\olu+ d\zeta\wedge\olu-X \cdot (\ole\cdot \olu-H\wedge
\olu)+H\wedge {(X\cdot \olu)} \qquad\text{(by \eqref{dHolu})}
\\
&=&\LieDer_a\olu+ d\zeta\wedge\olu-(X\wedge \ole ) \cdot \olu+X\cdot
(H\wedge
\olu)+H\wedge {(X\cdot \olu)}\\
&=&\LieDer_a\olu+ d\zeta\wedge\olu- (X\wedge \ole) \cdot
\olu+(\inserts_a+\zeta\wedge)(H\wedge \olu)
+H\wedge(\inserts_a \olu+\zeta\wedge\olu)\\
 &=&\LieDer_a\olu+ d\zeta\wedge\olu- (X\wedge \ole) \cdot
\olu+(\inserts_a H)\wedge \olu.
\end{eqnarray*}
This proves the last equation.\end{proof}

Another lemma needed is
\begin{lem}\label{Lem:Mukaivanish}\cite{MR2265463}
The Mukai pairing vanishes in $N_i\times N_k$, unless $i+k=2n$, in
which case it is nondegenerate.
\end{lem}

\begin{proof}[Proof of Proposition \ref{Thm:modularGC}.]
For an $X=a+\zeta\in \sections{\nolL}$, we assume that
$\barpd{(X\cdot \olu)}=f \olu$, for some function $f\in
C^\infty{(M,\CC)}$. Then Eq. (\ref{Eqt:LieDeroverlineV}) implies
that $ \LieDer_{X}\olV =(f-\duality{e}{X})\olV$ and hence
\begin{equation}\label{Eqt:LieDerlV}
\LieDer_{X} \V =( \duality{e}{X}-f) \V.
\end{equation}
We also have, according to Eq. (\ref{Eqt:LieDerau}) and
(\ref{Eqt:LieDeraoverlineu})
\begin{eqnarray*}
\LieDer_{\rho_\nolL(X)}s&=&\LieDer_a
\Mukai{u}{\olu}=\Mukai{\LieDer_a
u}{\olu}+\Mukai{u}{\LieDer_a\olu}\\
&=&\Mukai{\duality{e}{X} u -(d\zeta+\inserts_a H)\wedge u}{\olu}+
\Mukai{u}{f\olu+(d_{\olL}X )\cdot \olu -(d\zeta+\inserts_a H)\wedge\olu}\\
&=&( \duality{e}{X}+f)\Mukai{u}{\olu}=( \duality{e}{X}+f)s.
\end{eqnarray*}
In the last step, we have applied Eq. (\ref{Mukaequivariant}) and
Lemma  \ref{Lem:Mukaivanish} . In turn, we get
$$
\LieDer_{X} \V\otimes
s+\V\otimes\LieDer_{\rho_\nolL(X)}s=2\duality{e}{X}\V\otimes s.
$$
This proves the first claim. By symmetry, we also have
$$
\LieDer_{W} \olV\otimes
\bar{s}+\olV\otimes\LieDer_{\rho_{\olL}(W)}\bar{s}=
2\duality{\ole}{W}\olV\otimes \bar{s},
$$
for all $W\in\sections{\olL}$. By Eq. (\ref{Mukasymmetry}), we know
$$
\bar{s}=\Mukai{\olu}{u}=(-1)^n\Mukai{u}{\olu}=(-1)^n s.
$$
Thus there holds
$$
\LieDer_{W} \Omega\otimes
{s}+\Omega\otimes\LieDer_{\rho_{\olL}(W)}{s}=
2\duality{\ole}{W}\Omega\otimes {s},
$$
which implies that $2\ole$ is the modular cocycle of $\olL$ with
respect to $\Omega$ and $s$.
\end{proof}

%It is shown in \cite{} that there is an induced complex Poisson
%structure $\pi$ on $M$, which is characterized by XXX and XXX by
%\cite{}:
%$$
%\Poissonbracket{f,g}=\ip{\JJ df}{dg}.
%$$
%It is also proved that this $\pi$ is just the Poisson structure
%induced by the Lie bialgebroid  $(\nolL,\olL)$. By Corollary 2.8 in
%\cite{}, we have
%\begin{cor}The modular vector filed of the $(M,\pi)$ with respect to
%$s=\Mukai{u}{\bar{u}}$ is
%$$
%m_s= \rho_{\olL}(e)-\rho_{\nolL}(\ole)=pr_{{\TCM}}(e-\ole).
%$$
%\end{cor}

\section{Proof of  the main theorem}\label{Sec:ProofofMain}

We first finish the proof of Proposition
\ref{Prop:TwoMudlestrEquivalent}. \begin{proof}   $\olN$ has an
induced $\olL$-module structure arising from the $\olL$-module
$\module$ (see Eq. (\ref{6})). According to the second statement of
Proposition \ref{Thm:modularGC}, we know that this module structure
is determined by the following equation:
$$
\LieDer_{W}\olu=\duality{\ole}{W}\olu,\quad\forall~ W\in
\sections{\olL}.
$$
This just coincides with the standard $\olL$-module structure
defined by
 Eq. (\ref{Eqt:LieDerWolu}) because
 $$W\cdot\pd \olu=W\cdot ~\ole\cdot
\olu=\duality{\ole}{W} \olu,$$ by Lemma \ref{Lem:XcdotWcdot}.
\end{proof}
Now we are ready to prove the main theorem in this paper.
\begin{proof}[Proof of Theorem \ref{Thm:Main}]
 By Proposition \ref{Thm:modularGC} and  Eq. (\ref{19}), (\ref{20}),
  we conclude that
\begin{eqnarray*}
 \bdees (X\otimes \olu) &=&(d_{\olL }X   + \ole\wedge   X)\otimes
 \olu\,,\\
  \bdel (X\otimes \olu) &=&( -\del X+ \inserts_{e} X)\otimes
 \olu\,,
 \end{eqnarray*}
for all $X\in \sections{\wedge^k \nolL}$ .

Compare with the expression of $\pd$ in Eq. (\ref{Eqt:pdXbaru}), we
immediately know that Diagram (\ref{Digram1}) is commutative. To
prove the commutativity of Diagram (\ref{Digram2}), it suffices to
prove:
 \begin{equation}\label{Eqt:bdelbarpd}
\barpd(X\cdot \olu)=(-\del X+ \inserts_{e} X)\cdot \olu,\quad\forall
X\in
\sections{\wedge^i \nolL}.
\end{equation}

In fact, by Eq. (\ref{Eqt:pdbarwu}),
\begin{eqnarray*}
\mbox{LHS of Eq. (\ref{Eqt:bdelbarpd})
}&=&\barpd(X\cdot \olV\cdot u)=(-1)^{\frac{i(i-1)}{2}}\barpd((\inserts_X\olV)\cdot u)\\
&=& (-1)^{\frac{i(i-1)}{2}}(d_\nolL \inserts_X\olV+e\wedge
\inserts_X\olV)\cdot u .
\end{eqnarray*}
We also have
\begin{eqnarray*}
(\inserts_{e} X)\cdot \olu&=& e\cdot X\cdot \olV\cdot u  =
(-1)^{\frac{i(i-1)}{2}}e\cdot (\inserts_X\olV)\cdot
u=(-1)^{\frac{i(i-1)}{2}}(e\wedge \inserts_X\olV)\cdot u.
\end{eqnarray*}
And by Eq. (\ref{diverse}), (\ref{13bis}),
\begin{eqnarray*}
\del X&=&(-1)^{(i-1)(2n-1)}V^\sharp\Omega^\sharp\del
X\\
&=&(-1)^{(i-1)(2n-1)+i}V^\sharp d_\nolL \Omega^\sharp X =-V^\sharp
d_\nolL \Omega^\sharp X .
\end{eqnarray*}
Hence
\begin{eqnarray*}
-(\del X)\cdot \olu&=&(V^\sharp d_\nolL \Omega^\sharp
X)\cdot \olu\\
&=&(-1)^{\frac{(2n-i+1)(2n-i)}{2}}(d_\nolL\inserts_X \Omega)\cdot
V\cdot \olu=(-1)^{\frac{i(i-1)}{2}}(d_\nolL \inserts_X\olV)\cdot u.
\end{eqnarray*}
This proves Eq. (\ref{Eqt:bdelbarpd}), and the proof of Theorem
\ref{Thm:Main} is thus completed.
\end{proof}

\section{Some  corollaries}
The first obvious result is that, by the isomorphisms
$$\wedge^k \nolL\otimes\module\cong (\wedge^k \nolL)\cdot
\olN=\olN_k=N_{2n-k},
$$
the Dirac generating operator constructed by Theorem \ref{Thm:A}
  for $E=\nolL\oplus\olL$ is exactly
$$
\bdirac=\bdees+\bdee=\pd+\barpd=\dH,
$$
and especially $\fsmile=\bdirac^2=0$.

In Section \ref{subsec:ModularLiealgebroids}, we defined a pair of
operators $\dees$ and $\del$  on $\sections{\wedge A}$, and
similarly $\dee$ and $\dels$  on $\sections{\wedge A^*}$, for any
Lie bialgebroid $(A,A^*)$. Let $\Dee=\dees+\del$,
$\Dees=\dee+\dels$. Their squares yield the pair of
\textbf{Laplacian operators}\cite{AlekseevXu}
\begin{gather}
\lap=\Dee^2=\dees\del+\del\dees:~~ \sections{\wedge^k A}
\to\sections{\wedge^k A}, \label{g12} \\
\laps=\Dees^2=\dee\dels+\dels\dee:~~ \sections{\wedge^k
A^*}\to\sections{\wedge^k A^*}. \label{g13}
\end{gather}
As an immediate corollary of the following result
\begin{thm}\label{thmlaplacian}\cite[Theorem 3.4]{ChenStienon} If $(A,\As)$ is a Lie bialgebroid,
then
 $$\lap=\thalf(\ld{X_0}+\ld{\xi_0}) :~\sections{\wedge A}\to\sections{\wedge
A},$$
%\text{ in } \End_{\CC}\big(\sections{\wedge A}\big)$
$$\laps=\thalf(\ld{X_0}+\ld{\xi_0}) :~\sections{\wedge
A^*}\to\sections{\wedge A^*}.$$
\end{thm}
we have

\begin{cor}Let $(\nolL,\olL)$ be the Lie bialgebroid coming from a twisted
generalized complex structure $\JJ$. The Laplacian operators $\lap$
and $\laps$ defined by Eq. (\ref{g12}) and (\ref{g13}) are  given
respectively by
\begin{align}\nonumber
&\lap f=\laps f =\thalf pr_{T}(e+\ole)(f),\quad\forall f\in
\cinf{M,\CC},
\\\label{lap X}
&\lap X = \dbH{(\ole+\thalf e)}{X}-\frac{i}{2}\JJ(\dbH{e}{X}),
\quad\forall X\in \sections{\nolL},
\\\label{laps X}
&\laps W = \dbH{(e+\thalf
\ole)}{W}+\frac{i}{2}\JJ(\dbH{\ole}{W}),\quad\forall W\in
\sections{\olL}.
\end{align}
\end{cor}
\begin{proof}For $e\in \sections{\olL}$, $\ole\in \sections{\nolL}$,
the Lie derivations   $\LieDer_{e}$ and $\LieDer_{\ole}$ on
$\sections{\nolL}$ are given respectively by:
\begin{eqnarray*}
\LieDer_{e} X &=& pr_{\nolL}(\dbH{e}{X}),  \\
 \LieDer_{\ole} X &=& \dbH{\ole}{X},
\quad\forall X\in\sections{\nolL}.
\end{eqnarray*}

The projections of $\TC\oplus \TsC$ to $\nolL$ and $\olL$ are given
respectively by:
$$
pr_{\nolL}(z)=\thalf(z-i\JJ z),\quad pr_{\olL}(z)=\thalf(z+i\JJ
z),\quad\forall z\in \TC\oplus \TsC.
$$
Thus Eq. (\ref{lap X}) and (\ref{laps X}) follow directly from
Theorem \ref{thmlaplacian}.
\end{proof}

It is well known \cite{MR1262213} that, if $\anchor$ and $\anchors$
denote the anchor maps of a   Lie bialgebroid $(A,A^*)$, the bundle
map \begin{equation}\label{pisharp} \pi\diese=
\anchor\rond(\anchors)^*:\TsCM\to \TCM \end{equation} defines a
(complex) Poisson structure on   $M$.

In particular, for the Lie bialgebroid $(\nolL,\olL)$ coming from
the twisted generalized complex structure $\JJ$, $-i\pi$ is a real
Poisson structure. In fact, up to a factor of $2$, it is given by
  \cite{MR2398985,math/0703298}
\begin{equation}\label{GualtieriP}
P(\xi,\eta)=\ip{\JJ \xi}{\eta}.
\end{equation}

Let us briefly recall the definition of the modular vector field of
a Poisson manifold $(M,\pi)$ \cite{MR1484598}. Let
$\omega\in\OO^{\TOP}(M)$ be a volume form, the modular vector field
with respect to $\omega$ is the derivation $X_{\omega}$ of the
algebra of functions $\cinf{M}$ characterized by
\begin{equation}
\label{tongtong}
\LieDer_{\pisharp(df)}\omega = X_{\omega}(f)\omega .\end{equation}

For the Poisson structure induced from  a Lie bialgebroid,  the
relation between modular cocycles and modular vector fields is
established in the following:
\begin{lem}\cite[Corollary 3.8]{ChenStienon}\label{Cor:extremal}
Let $M$ be an orientable manifold with volume form
$s\in\OO^{\TOP}(M)$ and let $(A,A^*)$ be a real Lie bialgebroid over
$M$ with associated Poisson bivector $\pi$ defined by Eq.
(\ref{pisharp}). Then the modular vector field of the Poisson
manifold $(M,\pi)$ with respect to $s$ is
\begin{equation}\label{Eqt:alluded} X_s=\thalf\big( \anchors(\xi_0)-\anchor(X_0) \big) ,\end{equation}
where $\xi_0$ and $X_0$ are modular cocycles defined by Eq.
(\ref{17}) and (\ref{18}) (choosing arbitrary $V$ and $\Omega$).
\end{lem}

Hence as an immediate consequence, we obtain the following result of
Gualtieri \cite[Proposition 3.27]{math/0703298}.
\begin{cor}The modular vector field of the Poisson structure $P$ defined
in Eq. (\ref{GualtieriP}) is given by
$$
\frac{i}{2} pr_{\TC}(e-\ole),
$$
with respect to the volume form $s=\Mukai{u}{\olu}$.
\end{cor}

%\begin{cor}
%\begin{eqnarray}
%\lap X &=& \db{(\ole+\thalf e)}{X}-\frac{i}{2}\JJ(\db{e}{X}),
%\quad\forall X\in \sections{\nolL},
%\\
%\laps W &=& \db{(e+\thalf
%\ole)}{W}-\frac{i}{2}\JJ(\db{\ole}{W}),\quad\forall W\in
%\sections{\olL}.
%\end{eqnarray}
%\end{cor}

%\section*{List of Publications}
%\renewcommand{\refname}{Publications}

%\bibliographystyle{plain}
%\bibliography{pharyngitis}

\end{document}